\documentclass{article}[12pt]

\usepackage{fullpage, verbatim,graphicx,hyperref}
\usepackage[stable]{footmisc}
\usepackage{amsmath, amssymb, amsthm, amsfonts}

\newtheorem{theorem}{Theorem}
\newtheorem{corollary}[theorem]{Corollary}
\newtheorem{lemma}[theorem]{Lemma}
\theoremstyle{definition}
\newtheorem{example}[theorem]{Example}

\newcommand{\R}{\mathbb{R}}
\newcommand{\C}{\mathbb{C}}

\newcommand{\D}{\mathbb{D}}
\newcommand{\dbar}{\overline{\partial}}

\renewcommand{\Re}{\mathrm{Re}\,}
\renewcommand{\phi}{\varphi}

\author{Charles Z. Martin}
\date{\today}
\title{Perturbations of the Dirichlet Problem and Error Bounds}

\begin{document}

\maketitle
\begin{abstract}
\noindent
  The Dirichlet problem on a bounded planar domain is more readily understood and solved for the Laplace operator than it is for a
  Schr\"odinger operator. When the potential function is small, we might hope to approximate the solution to the
  Schr\"odinger equation with the solution to the Laplace equation. In this vein
  we develop a series expansion for the solution and give explicit bounds on the
  error terms when truncating the series. We also examine a handful of examples and
  derive similar results for the Green function and Dirichlet--to--Neumann map.
\end{abstract}

\section{Introduction}


Consider a boundary value problem with Dirichlet data, in the form
\begin{equation}\label{general_equation}
  \begin{cases}
    L\phi = 0 & \textrm{ in } D \\
    \phi = f & \textrm{ on } \partial D,
  \end{cases}
\end{equation}
where $L$ is a second--order, linear, elliptic differential operator and
$D$ is a bounded domain in the plane with sufficiently nice boundary. 
In trying to solve this problem,
there are a few potential obstacles; the domain could be difficult to handle, the operator might be overly complicated,
or---most commonly---a combination of the two. In this paper we discuss a method of approximating \eqref{general_equation}
that arises from viewing $L$ as a perturbation of the Laplacian and bound the resulting error. The approach is motivated as follows.

Suppose that $L = L_\epsilon$ denotes an elliptic operator depending on some perturbation parameter $\epsilon$,
with $L_0 = \Delta$. Let's denote the solution to \eqref{general_equation} by $\phi_\epsilon$; the unperturbed solution is denoted
$\phi_0$. These functions are identical on $\partial D$, so their difference satisfies
\begin{equation*}
  \begin{cases}
    L_\epsilon (\phi_\epsilon - \phi_0) = -L_\epsilon\phi_0 & \textrm{ in } D\\
    \phi_\epsilon - \phi_0 = 0 & \textrm{ on } \partial D.
  \end{cases}
\end{equation*}
Since $L_\epsilon\approx \Delta$, we have $-L_\epsilon \phi_0$ is very small---in fact, it vanishes to zeroth order in $\epsilon$.
What remains is an inhomogeneous problem with zero boundary data, which can be solved with knowledge of the Green function of
$L_\epsilon$. Thus we have reduced the problem of approximating $\phi_\epsilon$ to that of approximating a Green function.
A well--known, classical formula due to Hadamard (see \cite{Garabedian}, for example) approximates the change in a Green function due to
perturbations of the domain, but we'll instead make use of a result approximating the change arising from operator variation,
as detailed below.

Now let's be more specific. Consider a Schr\"odinger operator $L = \Delta - \epsilon u$, where $u\geq 0$ is a smooth scalar function.
Given a point $w\in D$, the Green function $g_w^*$ of $L$ is defined as the solution to
\begin{equation*}
  \begin{cases}
    Lg^*_w = \delta_w & \textrm{ in } D \\
    g^*_w = 0 & \textrm{ on } \partial D.
  \end{cases}
\end{equation*}
In the case $\epsilon=0$ we have $L=\Delta$ and denote the Green function simply by $g_w$. In other words, $g^*_w$
is the perturbation of $g_w$.
The following variational formula is derived, discussed and proved in \cite{Martin}.
\begin{theorem}\label{Schrodinger}
  Consider a bounded domain $D\subset \C$ with $C^1$ boundary and let $u\in C^\infty(\overline{D})$.
  Fix $w\in D$ and define the integral operator
  \begin{equation*}
    T\phi(z) = \int_D \phi u g_z\, dA.
  \end{equation*}
  The Green function $g_w^*$ of the Schr\"odinger operator $\Delta-\epsilon u$ satisfies
  \begin{equation*}
     g^*_w = g_w + \epsilon Tg_w + o(\epsilon)
  \end{equation*}
  as $\epsilon \to 0$, where the convergence of $o(\epsilon)$ is uniform.
  Furthermore, a full series expansion is given by
  \begin{equation} \label{Schrodinger_series_eq}
    g^*_w = \sum_{n=0}^\infty \epsilon^n T^n g_w.
  \end{equation}
\end{theorem}
With this result we can turn to the approximation of solutions to the Dirichlet problem. Before that, however,
we take a moment to discuss error bounds for the approximations of theorem \ref{Schrodinger}.

\section{Error Bounds for Variation of the Green Function}

Although the first order correction given in theorem \ref{Schrodinger} is correct as $\epsilon \to 0$, practical computation
might demand explicit error bounds for a fixed perturbation parameter. The rest of this section is devoted to the following theorem.
\begin{theorem}\label{Green_error}
  Let $D\subset \C$ be a bounded domain with $C^1$ boundary and let $u\in C^\infty(\overline{D})$. For fixed $w\in D$
  consider the Green function $g^*_w$ of the operator $\Delta-\epsilon u$ on the domain $D$ (with $\epsilon > 0$).
  The remainder $R_{n}$ in the expansion
  \begin{equation*}
     g^*_w = g_w + \epsilon Tg_w + \cdots + \epsilon^{n-1}T^{n-1} g_w + R_{n}
  \end{equation*}
  is uniformly bounded by
  \begin{equation*}
    \|R_{n}\|_\infty \leq \left(\frac{\epsilon \|u\|_\infty \mathrm{diam}(D)}{\sqrt{12}}\right)^n\frac{\mathrm{diam}(D)}{4\sqrt{3\pi}}.
  \end{equation*}
\end{theorem}
For simplicity, we often write $|R_n|$ in place of $\|R_n\|_\infty$.
Proof of this theorem requires bounding the tail of the series \eqref{Schrodinger_series_eq}, ultimately reducing to easily--understood
quantities such as the diameter of the domain. The following lemmas serve this purpose.
\begin{lemma} \label{T_bound}
  With the same assumptions as the previous theorem,
  given $\phi\in L^2(D)$ we have the pointwise bound $|T\phi(z)| \leq \|u\|_\infty \|g_z\|_2 \|\phi\|_2$. Furthermore,
  the operator norm of $T: L^2\to L^2$ satisfies
  \begin{equation}\label{norm_bound}
    \|T\| \leq \|u\|_\infty \left(\int_D \|g_\xi\|_2^2\, dA(\xi) \right)^{1/2} = \|u\|_\infty \left(\iint_{D^2} |g(\xi,\tau)|^2\,
    dA^2(\xi,\tau)\right)^{1/2}.
  \end{equation}
\end{lemma}
\begin{proof}
  The Cauchy--Schwarz inequality gives
  \begin{equation*}
    |T\phi(\xi)| = \left|\int_D ug_\xi\phi\, dA\right| \leq \|u\|_\infty \int_D |g_\xi\phi|\, dA \leq \|u\|_\infty \|g_\xi\|_2 \|\phi\|_2.
  \end{equation*}
  Integrating $|T\phi(\xi)|^2$ over $D$ yields
  \begin{equation*}
    \|T\phi\|_2 \leq \|u\|_\infty \|\phi\|_2 \left(\int_D \|g_\xi\|_2^2\, dA(\xi)\right)^{1/2},
  \end{equation*}
  from which the result follows.
\end{proof}
\begin{lemma}
    For $z\in \D$, the Green function $g_z$ of the Laplacian on $\D$ satisfies $\|g_z\|_2 \leq \|g_0\|_2 = (8\pi)^{-1/2}$.
  \end{lemma}
  \begin{proof}
    First note that $\|g_z\|_2$ only depends upon $r=|z|$ and that $\|g_1\|_2 = 0$, so we need to consider the integral
    \begin{equation*}
      \int_\D \ln^2 \left|\frac{w-r}{1-rw}\right|\, dA(w)
    \end{equation*}
    when $0\leq r<1$.
    The substitution $\xi = (w-r)/(1-rw)$ preserves the unit disk, giving
    \begin{equation*}
      \int_\D \ln^2 \left|\frac{w-r}{1-rw}\right|\, dA(w) = \int_\D \ln^2|\xi|  \frac{(1-r^2)^2}{|1+r\xi|^4}\, dA(\xi).
    \end{equation*}
    Consider the difference $\|g_0\|^2_2 - \|g_z\|^2_2$ as follows:
    \begin{align*}
      \|g_0\|^2_2 - \|g_z\|^2_2 &= \frac{1}{4\pi^2}\int_\D \ln^2|\xi|\left(1- \frac{(1-r^2)^2}{|1+r\xi|^4}\right)\, dA(\xi) \\
      &\geq  \frac{1}{4\pi^2}\int_\D \ln^2|\xi|\left(1- \frac{(1-r^2)^2}{1-r^4|\xi|^4}\right)\, dA(\xi) \\
      &= \frac{1}{4\pi^2}\int_\D \ln^2|\xi|\left(\frac{2r^2-r^4(1+|\xi|^4)}{1-r^4|\xi|^4}\right)\, dA(\xi).
    \end{align*}
    Since $0\leq r<1$ and $\xi\in\D$, we have $r^4(1+|\xi|^4) \leq 2r^2$ and the desired inequality follows. The computation of
    $\|g_0\|_2$ is a standard exercise in polar coordinates.
\end{proof}
\begin{lemma} \label{Diameter_bound}
  Let $D$ be as in theorem \ref{Schrodinger}. Then for $z\in D$ we have the bounds
  \begin{align*}
    \|g_z\|_2 \leq \frac{\mathrm{diam}(D)}{\sqrt{24\pi}} \quad \textrm{ and } \quad 
    \left(\iint_{D^2} |g(\xi,\tau)|^2 \, dA^2(\xi,\tau) \right)^{1/2} \leq \frac{\mathrm{diam}(D)}{\sqrt{12}}.
  \end{align*}
\end{lemma}
\begin{proof}
  First note that, by translation and dilation, the previous lemma can be generalized to a disk $D_R$ of radius $R$ centered at $p\in \C$.
  In this case, we have that
  \begin{equation*}
    \|g_\xi\|_2 \leq \|g_p\|_2 = \frac{R}{\sqrt{8\pi}}
  \end{equation*}
  for any $\xi\in D_R$. For such a disk we have that
  \begin{equation*}
    \iint_{D_R^2} |g(\xi,\tau)|^2 \, dA^2(\xi,\tau) = \int_{D_R} \|g_\xi\|_2^2 \, dA(\xi) \leq \int_{D_R} \frac{R^2}{8\pi} \, dA
    = \frac{R^2}{4}.
  \end{equation*}
  Recall the monotonicity of the Green function as a domain functional: for $\Omega_1 \subseteq \Omega_2$ we have
  \begin{equation*}
    g_{\Omega_2} \leq g_{\Omega_1} < 0
  \end{equation*}
  throughout $\Omega_1$. Thus if $D$ is contained within a disk of radius $R$, we have
  \begin{align*}
    \left(\iint_{D^2} |g_D(\xi,\tau)|^2 \, dA^2(\xi,\tau) \right)^{1/2}&\leq
      \left(\iint_{D^2} |g_{D_R}(\xi,\tau)|^2 \, dA^2(\xi,\tau) \right)^{1/2} \\
      &\leq \left(\iint_{{D_R}^2} |g_{D_R}(\xi,\tau)|^2 \, dA^2(\xi,\tau) \right)^{1/2} \\
      & \leq \frac{R}{2}.
  \end{align*}
  The results follow from Jung's theorem \cite{Jung}: a domain of diameter $d$ is contained in a disk of radius $d/\sqrt{3}$.
\end{proof}
\begin{proof}[Proof of Theorem \ref{Green_error}]
  We have that $R_n = \epsilon^n (T)^n g^*_w$. Lemma \ref{T_bound} gives
  \begin{equation*}
    |R_n| \leq \epsilon^n \|u\|_\infty \|g_w\|_2 \|T^{n-1} g_w^*\|_2 \leq \epsilon^n \|u\|_\infty \|g_w\|_2 \|T\|^{n-1}\|g_w^*\|_2.
  \end{equation*}
  Note that theorem \ref{Schrodinger} implies that $g_w \leq g_w^* \leq 0$, and hence $\|g_w^*\|_2 \leq \|g_2\|_2$.
  With lemmas \ref{T_bound} and \ref{Diameter_bound} we obtain
  \begin{equation*}
    |R_n| \leq \epsilon^n \|u\|_\infty \left(\frac{\mathrm{diam}(D)}{\sqrt{24\pi}}\right)^2
    \left(\|u\|_\infty \frac{\mathrm{diam}(D)}{\sqrt{12}}\right)^{n-1}
  \end{equation*}
  which rearranges into the result.
\end{proof}
\begin{example}
  Consider the problem of estimating the Green function $g_w^*$ of $\Delta - 1$ on the unit disk $\D$.
  Viewing $\Delta-1$ as a perturbation of $\Delta$ gives
  \begin{equation}\label{Green_perturb_example}
    g_w^*(z) = g_w(z) + \int_\D g_zg_w\, dA + R_2(z),
  \end{equation}
  with $R_2$ bounded by
  \begin{equation}\label{Green_bound_example}
    |R_2| \leq \frac{1}{6\pi\sqrt{3}} \simeq 0.0306.
  \end{equation}
  We can evaluate the integral in \eqref{Green_perturb_example} as follows. Note that $g_wg_z = \partial_n (g_w g_z) = 0$
  on $\partial \D$, so we have that
  \begin{align*}
    \int_\D g_zg_w\, dA &= \int_\D g_w(\xi) g_z(\xi) \Delta\left(\frac{|\xi|^2}{4}\right)\, dA(\xi) \\
    &= \int_\D \frac{|\xi|^2}{4}\Delta(g_z(\xi)g_w(\xi))\, dA(\xi) \\
    &= \frac{1}{4}\int_\D |\xi|^2 (g_z\delta_w + g_w \delta_z)\, dA(\xi)
      + \frac{1}{2}\int_\D |\xi|^2\nabla g_z(\xi)\cdot \nabla g_w(\xi)\, dA(\xi) \\
    &= g_w(z)\left(\frac{|z|^2+ |w|^2}{4}\right) + \frac{1}{2}\int_\D |\xi|^2\nabla g_z(\xi)\cdot \nabla g_w(\xi)\, dA(\xi).
  \end{align*}
  This last integral was evaluated in the final example of \cite{Martin}; when $zw\neq 0$ the result is
  \begin{equation*}
    \int_\D g_z g_w\, dA = g_w(z) \left(\frac{|z|^2+|w|^2}{4}\right) +
      \frac{1}{8\pi}\Re\left[\left(z\overline{w} + \frac{\overline{z}}{\overline{w}} + \frac{w}{z} - \frac{1}{z\overline{w}}\right)
        \log(1-z\overline{w}) - 2 z\overline{w}\log|z-w|\right]
  \end{equation*}
  When $w = 0$ we instead have
  \begin{equation*}
    \int_\D g_z g_0\, dA = \frac{1 - |z|^2 (1-\log|z|)}{8\pi}.
  \end{equation*}
  As a comparison, the actual Green function with singularity at 0 is given by
  \begin{equation*}
    g_0^*(z) = \frac{K_0(1) - K_0(|z|)}{2\pi},
  \end{equation*}
  where $K_0$ denotes the modified Bessel function of the second kind.
  The error $R_2(z)$ is given by
  \begin{equation}\label{Green_real_error}
    R_2(z) = \frac{K_0(1) - K_0(|z|)}{2\pi} - \frac{\log|z|}{2\pi} - \frac{1 -  |z|^2 (1 - \log|z|)}{8\pi},
  \end{equation}
  which takes values in the interval $[0,0.0088]$, confirming error bound \eqref{Green_bound_example}. A basic plot of $R_2$
  is shown in figure 1.
  \begin{figure}[h]
    \centering
    \includegraphics[width=0.75\textwidth]{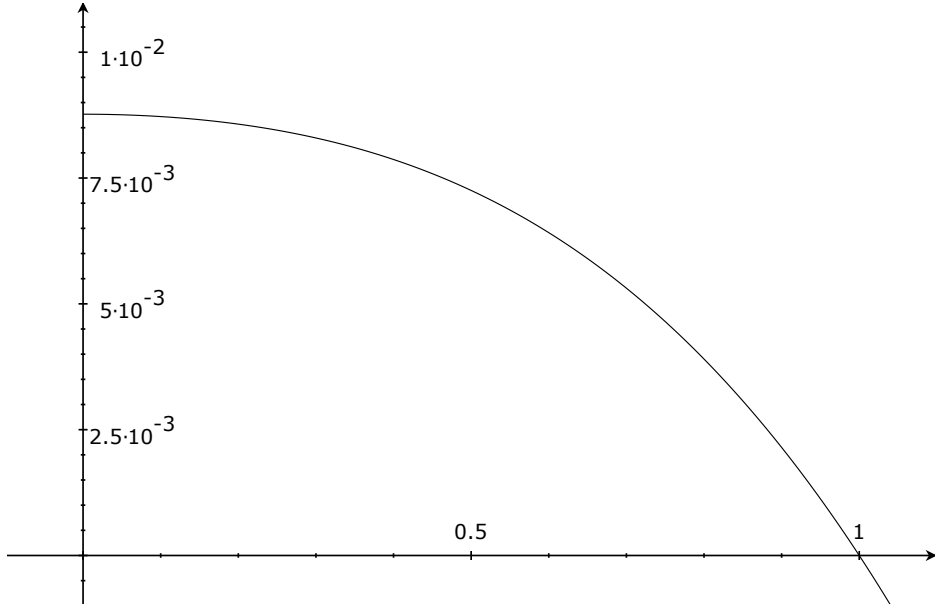}
    \caption{A plot of $R_2(|z|)$, as given in equation \eqref{Green_real_error}.}
  \end{figure}
\end{example}

\section{Application to Boundary--Value Problems}

\subsection{Variation of the Dirichlet Problem}
Consider solving the Dirichlet problem on a domain $D$:
\begin{equation} \label{Dirichlet}
  \begin{cases}
    (\Delta-\epsilon u) \phi_\epsilon = 0 & \textrm{ in } D \\
    \phi_\epsilon = f &\textrm{ on } \partial D
  \end{cases}
\end{equation}
where $f$ is a bounded Borel function, $u\geq 0$ is smooth in a neighborhood of $\overline{D}$
and $\epsilon \geq 0$ is small. If we ignored the $\epsilon u$ term
in the problem---that is, assume $\epsilon$ is zero---we can estimate the error with the following theorem and corollary.
\begin{theorem}\label{Dirichlet_variation}
  Let $D\subset \C$ be a bounded domain with $C^1$ boundary and suppose $u\in C^\infty(\overline{D})$ is a positive function.
  For $\epsilon \geq 0$ and a bounded Borel function $f$ the solution $\phi_\epsilon$ to the Dirichlet problem \eqref{Dirichlet} satisfies
  \begin{equation*}
    \phi_\epsilon(z) = \phi_0(z) + \epsilon \int_D u \phi_0 g_z\, dA + o(\epsilon)
  \end{equation*}
  as $\epsilon \to 0$, where the error term converges uniformly in $z$. Using the notation of theorem \ref{Schrodinger},
  a full series expansion is given by
  \begin{equation} \label{Schrodinger_series}
    \phi_\epsilon(z) = \sum_{n=0}^\infty \epsilon^n T^n \phi_0.
  \end{equation}
\end{theorem}
\begin{proof}
  Let $g^*_w$ denote the Green function of the operator $\Delta-\epsilon u$ in the region $D$.
  Note that
  \begin{equation*}
    \begin{cases}
      (\Delta-\epsilon u)(\phi_\epsilon - \phi_0) = \epsilon u \phi_0 & \textrm{ in } D \\
      \phi_\epsilon - \phi_0 = 0 & \textrm{ on } \partial D,
    \end{cases}
  \end{equation*}
  whence an expression for $\phi_\epsilon$
  is given by
  \begin{equation*}
    \phi_\epsilon(z) = \phi_0(z) + \epsilon\int_D u\phi_0 g^*_z\, dA.
  \end{equation*}
  From the perturbation formula \eqref{Schrodinger_series} we have
  \begin{equation*}
    \phi_\epsilon(z) = \phi_0(z) + \epsilon\int_D u\phi_0 g_z\, dA + o(\epsilon),
  \end{equation*}
  as desired. The full series expansion follows from instead using the full series given in theorem \ref{Schrodinger}.
\end{proof}
\begin{corollary}
  With the same assumptions as the previous theorem, the linearization of $\phi_\epsilon$ has the pointwise bound
  \begin{equation*}
    |\delta\phi_\epsilon(z)| \leq \|u\|_2\|g_z\|_2 \|f\|_\infty.
  \end{equation*}
\end{corollary}
\begin{proof}
  From the maximum modulus principle for harmonic functions, $\sup_D |\phi_0| \leq \sup_{\partial D} |f|$. The result follows from this and the
  Cauchy-Schwarz inequality.
\end{proof}
The following lemma aids computation of simple examples in the disk.
\begin{lemma}\label{Green_area}
  Given a point $z\in \D$ and an integer $n\geq 0$ we have
  \begin{equation*}
    \int_\D |\xi|^{2n} g_z(\xi)\, dA(\xi) = -\frac{1-|z|^{2n+2}}{4(n+1)^2},
  \end{equation*}
  where $g_z$ denotes the Green function of the Laplacian on $\D$ with singularity at $z$.
\end{lemma}
\begin{proof}
  First we note that
  \begin{equation*}
    \Delta\left(\frac{|\xi|^{2n+2}}{4(n+1)^2}\right) = 4\partial\dbar\left(\frac{(\xi\overline{\xi})^{n+1}}{4(n+1)^2}\right)
      = |\xi|^{2n}.
  \end{equation*}
  We use the Poisson--Jensen formula
  \begin{equation*}
    v(z) = \int_{\partial \D} v\partial_n g_z\, ds + \int_\D g_z\Delta v\, dA
  \end{equation*}
  to write
  \begin{align*}
    \frac{|z|^{2n+2}}{4(n+1)^2}
    &= \int_{\partial \D} \frac{|\zeta|^{2n+2}}{4(n+1)^2}\partial_n g_z(\zeta) ds(\zeta) + \int_\D |\xi|^{2n} g_z(\xi)\, dA(\xi) \\
    &= \int_{\partial\D} \frac{\partial_n g_z}{4(n+1)^2}\, ds + \int_\D |\xi|^{2n} g_z(\xi)\, dA(\xi) \\
    &= \frac{1}{4(n+1)^2} + \int_\D |\xi|^{2n} g_z(\xi)\, dA(\xi),
  \end{align*}
  where we've used the facts that $|\zeta|=1$ on $\partial \D$ and that $\partial_n g_z\, ds$ is a
  probability measure on $\partial \D$. We conclude that
  \begin{equation*}
    \int_\D |\xi|^{2n} g_z(\xi)\, dA(\xi) = -\frac{1-|z|^{2n+2}}{4(n+1)^2},
  \end{equation*}
  as desired.
\end{proof}

\begin{example}
  For a simple example, consider the Dirichlet problem for the Helmholtz equation on the unit disk $\D$:
  \begin{equation*}
  \begin{cases}
    (\Delta - a) \phi_a = 0 & \textrm{ in } \D \\
    \phi_a = 1 &\textrm{ on } \partial \D,
  \end{cases}
  \end{equation*}
  where $0 < a \ll 1$. Clearly $\phi_0\equiv 1$, so we obtain
  \begin{equation} \label{Dirichlet_example}
    \phi_a(z) = 1 + a\int_\D g_z\, dA + o(a),
  \end{equation}
  where the Green function $g_z$ is given by
  \begin{equation*}
    g_z(\xi) = \frac{1}{2\pi}\ln\left|\frac{\xi-z}{1-\overline{z}\xi}\right|.
  \end{equation*}
  Lemma \ref{Green_area} gives
  \begin{equation*}
    \int_\D g_z\, dA(\xi)
      = -\frac{1-|z|^2}{4},
  \end{equation*}
  so equation \eqref{Dirichlet_example} becomes
  \begin{equation*}
    \phi_a(z) = 1 - \frac{a}{4}(1-|z|^2) + o(a).
  \end{equation*}
\end{example}

\subsection{Explicit Error Bounds}

In this section we prove an analogue to theorem \ref{Green_error} for the perturbation of the Dirichlet problem considered above. Afterwards, we discuss a few examples to illustrate the approach's strengths and shortcomings.
\begin{theorem}\label{Dirichlet_error}
  Let $D\subset \C$ be a bounded domain with $C^1$ boundary and suppose $u\in C^\infty(\overline{D})$ is a positive function.
  For $\epsilon \geq 0$ and a bounded Borel function $f$ we can write the solution $\phi_\epsilon$ to the Dirichlet problem \eqref{Dirichlet}
  as
  \begin{equation*}
    \phi_\epsilon(z) = \phi_0(z) + \epsilon T\phi_0 + \cdots + \epsilon^{n-1} T^{n-1} \phi_0 + R_{n}(z),
  \end{equation*}
  where $T$ is as defined in theorem \ref{Schrodinger}. A uniform bound for the error $R_{n}$ is given by
  \begin{equation*}
    \|R_{n}\|_\infty \leq \left(\frac{\epsilon \|u\|_\infty \mathrm{diam}(D)}{\sqrt{12}}\right)^{n}
    \frac{\|f\|_\infty \sqrt{\mathrm{area}(D)}}{\sqrt{2\pi}}.
  \end{equation*}
\end{theorem}
\begin{proof}
  From theorem \ref{Dirichlet_variation} we have that
  \begin{equation*}
    R_n(z) = \epsilon^n T^n \phi_\epsilon (z).
  \end{equation*}
  Using the pointwise bound in lemma \ref{T_bound} gives
  \begin{align*}
    |R_n(z)| &\leq  \epsilon^n \|T^n \phi_\epsilon\|_\infty \\  
    &\leq \epsilon^n \|u\|_\infty \|T^{n-1}\phi_\epsilon\|_2 \sup_z \|g_z\|_2 \\
    &\leq \epsilon^n \|u\|_\infty \|T\|^{n-1} \|\phi_\epsilon\|_2 \sup_z \|g_z\|_2.
  \end{align*}
  Lemmas \ref{T_bound} and \ref{Diameter_bound} also yield
  \begin{equation*}
    \|T\| \leq \|u\|_\infty \left(\iint_{D^2} |g(\xi,\tau)|^2\, dA^2(\xi,\tau)\right)^{1/2} \leq
      \frac{\|u\|_\infty \mathrm{diam}(D)}{\sqrt{12}}.
  \end{equation*}
  Together with the bound on $\|g_z\|_2$ in lemma \ref{Diameter_bound} we have
  \begin{align*}
    |R_n(z)|
    &\leq \left(\frac{\epsilon\|u\|_\infty \mathrm{diam}(D)}{\sqrt{12}}\right)^{n} \frac{\|\phi_\epsilon\|_2}{\sqrt{2\pi}} \\
    &\leq \left(\frac{\epsilon \|u\|_\infty \mathrm{diam}(D)}{\sqrt{12}}\right)^{n}
    \frac{\|\phi_0\|_\infty \sqrt{\mathrm{area}(D)}}{\sqrt{2\pi}}.
  \end{align*}
  The result follows from the maximum principle for $\Delta - \epsilon u$; that is, $\|\phi_\epsilon\|_\infty \leq \|f\|_\infty$.
\end{proof}
\begin{corollary}\label{Dirichlet_disk_error}
  In the case when $D$ is a disk of radius $r$ the error bounds can be tightened into
  \begin{equation*}
    \|R_{n}\|_\infty \leq \left(\frac{\epsilon r \|u\|_\infty}{2}\right)^{n}
    \frac{r \|f\|_\infty}{\sqrt{2}}.
  \end{equation*}
\end{corollary}
\begin{example}\label{dirichlet_example}
  We return to the simple case of the previous example, the Dirichlet problem
  \begin{equation*}
  \begin{cases}
    (\Delta - a) \phi_a = 0 & \textrm{ in } \D \\
    \phi_a = 1 &\textrm{ on } \partial \D.
  \end{cases}
  \end{equation*}
  If we take $a=1$ the solution is known to be $\phi_1(z) = I_0(|z|)/I_0(1)$, where $I_0$ denotes the zeroth order modified Bessel function of the
  first kind. The function $\phi_1$ is an increasing function of $|z|$, and takes values in the interval $[0.789,1]$.
  Using corollary \ref{Dirichlet_disk_error}, the zeroth order approximation $\phi_1(z) \simeq \phi_0(z) \equiv 1$ has error bounded by
  \begin{equation*}
    |\phi_1 - 1| = |R_1| \leq \left(\frac{\epsilon r \|u\|_\infty}{2}\right)^{n}
    \frac{r \|f\|_\infty}{\sqrt{2}} = \frac{1}{\sqrt{8}} \simeq 0.354.
  \end{equation*}
  The actual error is
  \begin{equation*}
    R_1(x) = 1 - \frac{I_0(x)}{I_0(1)}, 
  \end{equation*}
  a plot of which appears in figure 2.
  The first order approximation is computed in the previous example as
  \begin{equation*}
    \phi_1(z) \simeq 1 - \frac{1 - |z|^2}{4},
  \end{equation*}
  with an error bound of
  \begin{equation*}
    |R_2|\leq \frac{1}{4\sqrt{2}} \simeq 0.177.
  \end{equation*}
  Indeed, $R_2(z) = \phi_1(z) - (1 - \frac{1 - |z|^2}{4})$ takes values in the interval $[0, 0.0399]$, as seen in figure 2.
  Taking it one step further, we can use lemma \ref{Green_area} to compute the next order term:
  \begin{equation*}
    \phi_1(z) \simeq 1 - \frac{1-|z|^2}{4} - G\left(\frac{1-|z|^2}{4}\right) = 1 - \frac{1-|z|^2}{4} + \frac{|z|^4 - 4|z|^2+3}{64}.
  \end{equation*}
  The estimated error is
  \begin{equation*}
    |R_3|\leq \frac{1}{8\sqrt{2}} \simeq 0.0884,
  \end{equation*}
  though the real error is bounded by $0.0071$ (we chose not to plot $R_3$ alongside the other error functions for visibility reasons).
  Even though the error bound decreases geometrically, it is far from tight in these first few cases.
  \begin{figure}[h]
    \centering
    \includegraphics[width=0.75\textwidth]{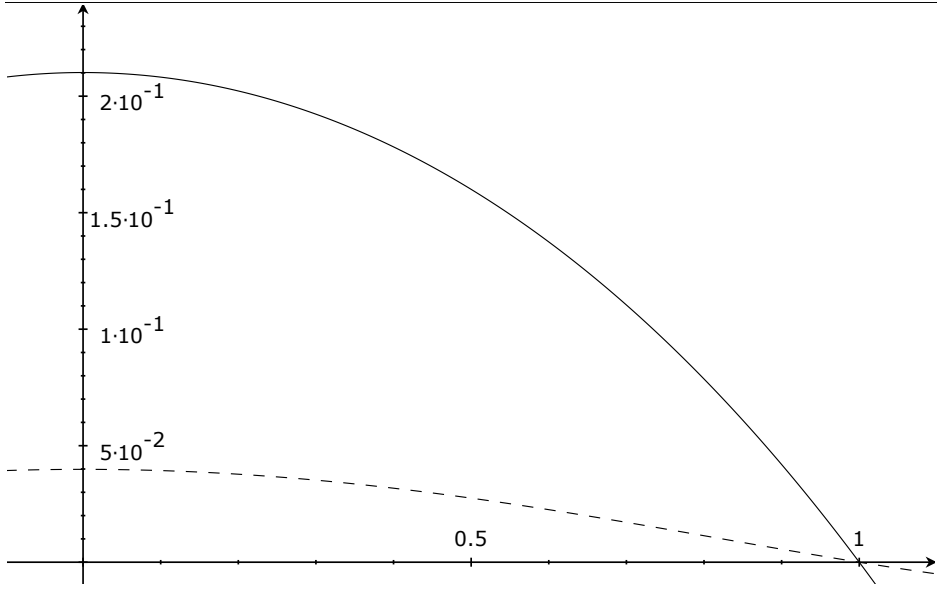}
    \caption{A plot of $R_1(|z|)$ and $R_2(|z|)$ from example \ref{dirichlet_example}. The solid line depicts $R_1$.}
  \end{figure}
\end{example}
\begin{example}
  Continuing in a similar fashion to the previous example, consider the Dirichlet problem
  \begin{equation*}
  \begin{cases}
    (\Delta - \epsilon |z|^2) \phi_\epsilon = 0 & \textrm{ in } \D \\
    \phi_\epsilon = 1 &\textrm{ on } \partial \D.
  \end{cases}
  \end{equation*}
  The solution to this problem when $\epsilon = 1$ is $\phi_1(z) = I_0(|z|^2/2)/I_0(1/2)$. In particular, $\phi_1(\D) \subseteq [0.94,1]$,
  which is a smaller range than before.
  In this situation the error in approximating $\phi_1$ by $\phi_0\equiv 1$ is less than $0.6$;
  without this knowledge, the bound of corollary \ref{Dirichlet_disk_error} gives
  \begin{equation*}
    |\phi_1-1| = |R_1| \leq \frac{1}{2\sqrt{2}} \simeq 0.354,
  \end{equation*}
  the same bound as in the previous example. Unfortunately, theorem \ref{Dirichlet_error} cannot detect that $\Delta - |z|^2$ is in some
  sense closer to $\Delta$ than $\Delta - 1$ is. This is due to the fact that $\||z|^2\|_\infty = \|1\|_\infty = 1$ on $\D$.
\end{example}
\begin{example}
  For a less trivial example we turn to the Helmholtz operator on an ellipse of small eccentricity. Consider the Dirichlet problem
  \begin{equation*}
  \begin{cases}
    (\Delta - a) \phi_a = 0 & \textrm{ in } E \\
    \phi_a = 1 &\textrm{ on } \partial E
  \end{cases}
  \end{equation*}
  on the ellipse $E$ given by
  \begin{equation*}
    x^2+\frac{y^2}{1.1^2} \leq 1.
  \end{equation*}
  Here we use the standard identification $z=x+iy$, with $x,y\in \R$. Once again note that $\phi_0\equiv 1$ and that $\phi_1(z) \simeq 1$
  with error bounded by
  \begin{equation*}
    |\phi_1 - 1| = |R_1| \leq \frac{2.2\sqrt{1.1}}{\sqrt{24}} \simeq 0.471.
  \end{equation*}
  In particular, $0.529 \leq \phi(0,0) \leq 1$.
  To compute the next order approximation, we have the following lemma.
  \begin{lemma}
    For the elliptical region $E$ given by $x^2/a^2+y^2/b^2\leq 1$ and a point $w=u+iv \in E$ we have
    \begin{equation*}
      \int_E g_w\, dA = \frac{(bu)^2 + (av)^2 - (ab)^2}{2(a^2+b^2)},
    \end{equation*}
    where $g_w$ denotes the Green function of the Laplacian on $E$.
  \end{lemma}
  \begin{proof}
    Since both $g_w$ and $x^2/a^2+y^2/b^2-1$ vanish on $\partial E$, Green's identity gives
    \begin{align*}
      \left(\frac{2}{a^2} + \frac{2}{b^2}\right)\int_E g_w \, dA &=
      \int_E g_w(x,y)\Delta\left(\frac{x^2}{a^2}+\frac{y^2}{b^2}-1\right)\, dA(x,y) \\
      &= \int_E \left(\frac{x^2}{a^2}+\frac{y^2}{b^2}-1\right)\Delta g_w(x,y)\, dA(x,y) \\
      &= \frac{u^2}{a^2}+\frac{v^2}{b^2}-1.\qedhere
    \end{align*}
  \end{proof}
  With lemma in hand we can compute
  \begin{equation*}
    \phi_1(x,y) = 1 + G1 + R_2 = 1 + \frac{1.21x^2+ y^2 - 1.21}{4.42} + R_2(x,y),
  \end{equation*}
  with $R_2$ bounded by
  \begin{equation*}
    |R_2| \leq \frac{2.2^2\sqrt{1.1}}{12\sqrt{2}} \simeq 0.299.
  \end{equation*}
  In particular, $0.427\leq \phi(0,0) \leq 1$. At the origin, our approximation gives a weaker bound
  than the one deduced from the first approximation.
\end{example}

\section{The Dirichlet--to--Neumann Map}
The Dirichlet--to--Neumann map $\Lambda_u: H^{1/2}(\partial D) \to H^{-1/2}(\partial D)$ is defined by
$\Lambda_u(f) = \partial_n \phi$, where $\phi$ solves the Dirichlet problem
\begin{equation*}
  \begin{cases}
    (\Delta - u)\phi = 0 & \textrm{ in } D\\
    \phi= f & \textrm{ on } \partial D.
  \end{cases}
\end{equation*}
In practical applications the Dirichlet--to--Neumann map is readily determined via boundary measurements at $\partial D$.
For this reason, there is much interest in inverse problems of the following type: given $\Lambda_u$, can we determine $u$?
Traditionally more attention is paid to the Laplace--Beltrami inverse problem---that is, determine $\lambda$ from knowledge of the
Dirichlet--to--Neumann map of $\nabla \lambda \nabla$. We will restrict our attention to the Schr\"odinger operator,
as an application of the theorems in the previous sections.

In his brief but seminal work on the topic \cite{Calderon},
A. P. Calder\'on considers the quadratic form arising from the self--adjoint operator $\Lambda$,
computes its linearization, and shows that the resulting linear map is injective. In a similar spirit we can compute the linearization
of $\Lambda$ using theorem \ref{Dirichlet_variation}.
We'll make use of the following lemma, which is proved in \cite{Martin2}.
\begin{lemma}
  Let $D$ be a bounded domain in $\C$ with smooth, analytic boundary and suppose that $f\in C^1(\overline{D})\cap C^2(D)$.
  If $f=0$ on $\partial D$ then for $\zeta\in \partial D$,
  \begin{equation*}
    \frac{\partial f}{\partial n}(\zeta) = \int_D P_\zeta \Delta f\, dA,
  \end{equation*}
  where $P_\zeta(z) = \partial_n g_z(\zeta)$ is the Poisson kernel of $D$.
\end{lemma}
With this lemma we can derive an expression for the Dirichlet--to--Neumann map for $\Delta-\epsilon u$.
For given Dirichlet boundary data $f$, define $\phi_\epsilon$ as the solution to the problem
\begin{equation*}
  \begin{cases}
    (\Delta - \epsilon u)\phi = 0 & \textrm{ in } D\\
    \phi= f & \textrm{ on } \partial D.
  \end{cases}
\end{equation*}
Note that $\phi_0$ is the solution to the problem for the Laplace operator.
\begin{theorem}
  For a bounded domain $D\subset \C$ with smooth, analytic boundary we have
  \begin{equation*}
    (\Lambda_{\epsilon u} f)(\zeta) = (\Lambda_0 f)(\zeta) + \epsilon \int_D u P_\zeta \phi_0\, dA + o(\epsilon).
  \end{equation*}
  An alternative expression is given by
  \begin{equation*}
    (\Lambda_{\epsilon u} f)(\zeta) = (\Lambda_0 f)(\zeta)
      + \epsilon \int_{\partial D} f(\xi) \left(\int_D uP_\xi P_\zeta\, dA\right)\, ds(\xi) + o(\epsilon).
  \end{equation*}
\end{theorem}
\begin{proof}
  Note that $\phi_\epsilon - \phi_0$ vanishes on $\partial D$. The lemma above gives
  \begin{equation*}
    (\Lambda_{\epsilon u} f)(\zeta) - (\Lambda_0 f)(\zeta) = \partial_n(\phi_\epsilon - \phi_0)(\zeta)
    = \int_D P_\zeta \Delta(\phi_\epsilon - \phi_0)\, dA = \int_D uP_\zeta \phi_\epsilon\, dA.
  \end{equation*}
  Since $\phi_\epsilon = \phi_0 + o(\epsilon)$, we obtain
  \begin{equation*}
    (\Lambda_{\epsilon u} f)(\zeta) = (\Lambda_0 f)(\zeta) + \epsilon \int_D u P_\zeta \phi_0\, dA + o(\epsilon).
  \end{equation*}
  Finally, note that $\phi_0$ can be written as a Poisson integral; this gives
  \begin{equation*}
    \int_D u P_\zeta \phi_0\, dA = \int_D u(\tau) P_\zeta(\tau)  \left(\int_{\partial D} P_\xi(\tau) f(\xi)\, ds(\xi)\right)\, dA(\tau)
    = \int_{\partial D} f(\xi)\left( \int_D u P_\xi P_\zeta\, dA\right)\, ds(\xi). \qedhere
  \end{equation*}
\end{proof}

\bibliographystyle{amsplain}

\vspace{0.2in}
\textsc{Department of Mathematics, Vanderbilt University, Nashville, TN, 37240} \\
\url{charles.z.martin@vanderbilt.edu}

\end{document}